\documentclass{bmcart}

\usepackage{amssymb,amsmath,amsthm, hyperref, enumerate, genyoungtabtikz, ytableau, tikz}
\usepackage[utf8]{inputenc} 

\def\includegraphics{}

\startlocaldefs

\newtheorem{thm}{Theorem}[section]
\newtheorem*{thmpart}{Theorem}
\newtheorem{df}{Definition}
\newtheorem*{rem}{Remark}
\newtheorem*{ex}{Example}
\newtheorem{pr}[thm]{Proposition}
\newcommand{\Z}{\mathbb{Z}}
\newcommand{\Q}{\mathbb{Q}}
\newcommand{\C}{\mathbb{C}}
\renewcommand{\H}{\mathbb{H}}
\newcommand{\calP}{\mathcal{P}}
\newcommand{\sgn}{\mathrm{sgn}}
\newcommand{\calQ}{\mathcal{Q}}
\newcommand{\leg}[2]{\genfrac{(}{)}{}{}{#1}{#2}}
\newcommand{\qbrack}[1]{\langle #1 \rangle_q}

\endlocaldefs

\begin{document}

\begin{frontmatter}

\begin{fmbox}
\dochead{Research}


\title{On $p$-adic modular forms and the Bloch-Okounkov theorem}


\author[
   addressref={aff1},                   
   email={mjg4@princeton.edu}   
]{\inits{MJG}\fnm{Michael J.} \snm{Griffin}}
\author[
   addressref={aff2},
   email={mjameson@utk.edu}
]{\inits{MJ}\fnm{Marie} \snm{Jameson}}
\author[
   addressref={aff3},
   corref={aff3},
   email={strebat@emory.edu}
]{\inits{STL}\fnm{Sarah} \snm{Trebat-Leder}}

\address[id=aff1]{
  \orgname{Department of Mathematics, Princeton University}, 
  \city{Princeton, New Jersey},                              
  \postcode{08544}                                
  \cny{USA}                                    
}
\address[id=aff2]{
  \orgname{Department of Mathematics, University of Tennessee},
  \city{Knoxville, Tennessee}
  \postcode{37996}
  \cny{USA}
}
\address[id=aff3]{
  \orgname{Department of Mathematics and Computer Science, Emory University},
  \city{Atlanta, Georgia}
  \postcode{30322}
  \cny{USA}
}


\begin{artnotes}
\end{artnotes}

\end{fmbox}


\begin{abstractbox}

\begin{abstract} 
Bloch-Okounkov studied certain functions on partitions $f$ called shifted symmetric polynomials. They showed that certain $q$-series arising from these functions (the so-called \emph{$q$-brackets} $\left<f\right>_q$) are quasimodular forms.  We revisit a family of such functions, denoted $Q_k$, and study the $p$-adic properties of their $q$-brackets.  To do this, we define regularized versions $Q_k^{(p)}$ for primes $p.$  We also use Jacobi forms to show that the $\left<Q_k^{(p)}\right>_q$ are quasimodular and find explicit expressions for them in terms of the $\left<Q_k\right>_q$.

\end{abstract}




\end{abstractbox}
%

\end{frontmatter}



\section{Introduction and statement of results}

In \cite{Serre1973}, Serre introduced the theory of $p$-adic modular forms, which are $p$-adic limits of compatible families of $q$-expansions of classical level one modular forms.  The first example of this phenomenon comes from the Eisenstein series \[G_k(\tau) := -\frac{B_k}{2k} + \sum_{n = 1}^\infty \sigma_{k - 1}(n)q^n \in \widetilde{M}_k,\] where $k$ is a positive even integer, $\tau \in \H,$ $q = e^{2 \pi i \tau}$, $B_k$ is the $k$th Bernoulli number, $\sigma_{k - 1}(n)$ is the sum of the $k - 1$ powers of the divisors of $n$, and $\widetilde{M}_k$ is the space of weight $k$ quasimodular forms. For primes $p \geq 5$, we regularize $G_k$ to 
\[G_k^{(p)}(\tau) := \frac{-B_k^{(p)}}{2k} + \sum_{n \geq 1} \sigma_{k - 1}^{(p)}(n)q^n = G_k(\tau) - p^{k - 1}G_k(p\tau) \in \widetilde{M}_k(p),\]
where
\[\sigma_{k - 1}^{(p)}(n) := \sum_{d \mid n, (d, p) = 1} d^{k - 1} ,\; \; \text{ and } \; \; 
B_k^{(p)} := (1 - p^{k - 1})B_k.\]
In order to find congruences for these regularized Eisenstein series, we recall Euler's theorem, which says that if $(a, n) = 1$ then $a^{\phi(n)} \equiv 1 \pmod{n},$ where
\[\phi(n):= \#\{k \in \Z, 1 \leq k \leq n, (n, k) = 1\}\]
is Euler's totient function. Together, Euler's theorem and the Kummer congruences for the Bernoulli numbers give the following: if \mbox{$k_1, k_2 \not \equiv 0 \pmod{p - 1}$}, then \[G_{k_1}^{(p)} \equiv G_{k_2}^{(p)} \pmod{p^r} \; \; \text{ whenever }\; \;  k_1 \equiv k_2 \pmod{\phi(p^r)}.\] Since we also have that $G_{k}^{(p)} \equiv G_k \pmod{p^r}$ whenever $k > r$, this makes \[G_k^{(p)} = \lim_{r \to \infty} G_{k + \phi(p^r)}\] into a $p$-adic modular form for $k \not \equiv 0 \pmod{p - 1}$.  Note that if \mbox{instead} $k \equiv 0 \pmod{p - 1}$, then the constant term is not $p$-integral, but the normalized Eisenstein series $E_k(\tau)$ satisfies
$E_k(\tau) \equiv 1 \pmod{p^r}$ whenever $k \equiv 0 \pmod{\phi(p^r)}$. 

Katz \cite{Katz73} and others have reformulated and expanded this theory to consider $p$-adic modular forms as $p$-adic analytic functions on elliptic curves. However in this paper we will only consider $p$-adic modular forms in the sense of Serre.

In this article, we wish to examine the $p$-adic properties of certain quasimodular forms $\qbrack{Q_k}$ and show that they are in many ways analogous to the Eisenstein series $G_k$ and fit into Serre's framework. 

Let $\calP$ be the set of all integer partitions. For any function $f: \calP \to \Q$, we define the ``$q$-bracket of $f$'' to be the following formal power series obtained by ``averaging'':
\begin{equation}
\label{Q_Brak_Def}
\left<f\right>_q := \frac{\sum_{\lambda \in \calP} f(\lambda) q^{|\lambda|}}{\sum_{\lambda \in \calP} q^{|\lambda|}}= \eta(\tau)\sum_{\lambda \in \calP} f(\lambda) q^{|\lambda| -1/24},\end{equation}
where $|\lambda|$ denotes the size of a partition $\lambda$ and $\eta(\tau)$ is Dedekind's eta-function.

A result of Bloch and Okounkov \cite{BlochOkounkov2000} gives that for a large class of functions $f: \calP \to \Q$, called shifted symmetric polynomials, the $q$-series $\left<f\right>_q$ is a quasimodular form on the full modular group. The space of shifted symmetric polynomials is generated by functions $Q_k$, which we will define explicitly in Section \ref{definitions}; the first few are given by \[Q_0(\lambda) = 1, \,Q_1(\lambda) = 0, \,Q_2(\lambda) = |\lambda| - 1/24, \,\dots.\]  Note that we use the same letter $Q_k$ to denote both this function and the corresponding generator of the formal polynomial algebra \[\mathcal{R} = \Q[Q_1, Q_2, \dots]\] in infinitely many variables.  Then any element \[f = f(Q_1, Q_2, \dots) \in \mathcal{R}\] can be considered as a function on partitions by setting \[f(\lambda) = f(Q_1(\lambda), Q_2(\lambda), \dots)\] and we can speak of its $q$-bracket $\left<f\right>_q$.  We give $\mathcal{R}$ a grading by assigning to $Q_k$ the weight $k$. Then the Bloch-Okounkov Theorem states that for $f \in \mathcal{R}$ homogeneous of grading $k$, $\left<f\right>_q$ is a quasimodular form of weight $k$ on the full modular group.  In their paper, Bloch and Okounkov also defined so-called $n$-point correlation functions which are related to $\left<Q_k\right>_q,$ and gave a formula for them involving derivatives of a theta function. We will use a special case of this result in equation \eqref{BOThm6.5} of Section \ref{jacobi}.

Zagier revisited this work in \cite{Zagier2015}, giving a significantly shorter proof of the Bloch-Okounkov Theorem and studying additional properties of the $q$-bracket.  

In Section~\ref{definitions}, we will define a normalization $\calQ_k$ of $\qbrack{Q_k}$ and a regularization $\calQ_k^{(p)}$ of $\calQ_k$ at the prime $p \geq 5$ with the property that 
\begin{equation}\label{congruence_remark}
\calQ_k^{(p)} \equiv \calQ_k \pmod{p^{k-1}}.
\end{equation}
 We will show the following about these $q$-series.

\begin{thm}  
\label{mainthm} Let $p \geq 5$ be prime.
\begin{enumerate}[(a)]
\item If $k_1, k_2 \not \equiv 0 \pmod{p - 1}$, then $\calQ_{k_1}^{(p)} \equiv \calQ_{k_2}^{(p)} \pmod{p^r}$ whenever $k_1 \equiv k_2 \pmod{\phi(p^r)}.$
\item If $k \not \equiv 0 \pmod{p - 1}$, then $\calQ_{k}^{(p)}$ is a $p$-adic modular form.
\item If $p > k$, then the modulo $p$ filtration of $\calQ_k^{(p)}$ (and $\calQ_k$) is $k(p + 1)/2$.  
\item The $q$-series $\calQ_k^{(p)}$ is a quasimodular form of weight $k$ on $\Gamma_0(p^2)$.
\item We have that $\calQ_k^{(p)}(\tau) = \calQ_k(\tau) - p^{k - 1}\calQ_k(p^2 \tau) - p^{k - 1}f_k^{(p)},$ where $f_k^{(p)}$ is given explicitly in Section \ref{jacobi}. In particular, $f_k^{(p)}$ is supported on $q^N$ with $\leg{2}{p} = \leg{N}{p}$. 
\end{enumerate}
\end{thm}

\begin{rem}
In \cite{Lopez2015}, Lopez studied the functions $\left<Q_{3}^{2n}\right>_q$ and showed that they also satisfy parts (a) and (b) of our theorem. It seems likely that other products of the $Q_k$ yield quasimodular forms which satisfy similar $p$-adic properties. 
\end{rem}

\begin{ex}
Consider the case where $p = 5$.  We have that the normalized $q$-bracket $\calQ_2$ and its regularization $\calQ_2^{(5)}$ are
\begin{align*}
\calQ_2 =& \boxed{\frac{-1}{24}} + q + 3q^2 + \boxed{4q^3} + 7q^4 + 6q^5 + 12q^6 + \boxed{8q^7} + \boxed{15q^8} + 13q^9 \\ &+ O(q^{10})  \in \widetilde{M}_2 \text{ and}\\
\calQ_2^{(5)} =& \boxed{\frac{1}{6}} + q + 3q^2 + \boxed{-1q^3 }+ 7q^4 + 6q^5 + 12q^6 + \boxed{13q^7} + \boxed{0q^8} + 13q^9  \\ &+ O(q^{10})  \in \widetilde{M}_2(25).\\
\end{align*}
Here, we have marked the terms above whose coefficients disagree; in accordance with part (e) of Theorem \ref{mainthm}, we see that this occurs only when $\leg{N}{5} = 0, -1$.

Since $2 \equiv 22 \pmod{\phi(25)}$, we also give the weight $22$ forms, which part (e) guarantees must differ from each other on the same powers of $q$:

\begin{align*}
\calQ_{22} &= \boxed{-162912981133/552 }+ q + 10460353203q^2 + \boxed{476837158203124q^3} \\
& +
558545864083284007q^4 + 109418989121052006006q^5\\  & + 
7400249944258160101212q^6 
 +\boxed{ 247064528596613234501288q^7} \\ & +
\boxed{4987885095119476318359375q^8} + 69091933354462879257896413q^9 \\ & + O(q^{10})  \in \widetilde{M}_{22} \text{ and}\\
\calQ_{22}^{(p)} &= \boxed{19420740739464719098414873/138} + q + 10460353203q^2 +\boxed{- q^3} \\ &+
558545864083284007q^4 + 109418989121052006006q^5 \\ & + 
7400249944258160101212q^6 + \boxed{247064529073450392704413q^7} \\ & +
\boxed{0q^8} + 69091933354462879257896413q^9 + O(q^{10})  \in \widetilde{M}_{22}(25).\\ 
\end{align*}
Notice that in accordance with part (a) of Theorem~\ref{mainthm} and (\ref{congruence_remark}), we have that \[\calQ_{22} \equiv \calQ_{22}^{(5)} \equiv \calQ_2^{(5)} \pmod{25}.\]
\end{ex}

In Section \ref{definitions}, we will explicitly define the functions $Q_k:\calP\rightarrow\mathbb{Q},$ as well as $\calQ_{k}^{(p)}$ and various other functions. In Section \ref{padic_section}, we will prove parts (a)-(c) of Theorem~\ref{mainthm}.  In Section \ref{jacobi}, we will make a connection to the theory of Jacobi forms in order to prove parts (d) and (e) of Theorem \ref{mainthm}. 

\section{Preliminary Definitions} \label{definitions}
\subsection{Definitions of $P_k(\lambda), Q_k(\lambda)$, and $\calQ_k(\tau)$} First, we must define the functions $Q_k:\calP\rightarrow\Q$ of Bloch and Okounkov (as described by Zagier in \cite{Zagier2015}).  For a partition $\lambda \in \calP,$ we first consider the Frobenius coordinates of $\lambda$, which are given by
\[(r; a_1, \ldots, a_r; b_1, \ldots, b_r),\]
where $r$ is the length of the longest diagonal in the Young diagram of $\lambda$ (i.e., $r$ is the size of the Durfee square of $\lambda$) and $a_1, \ldots a_r$ (resp., $b_1, \ldots, b_r$) are the arm-lengths (resp., leg-lengths) of the cells on this diagonal.  For example, the partition $\lambda = (4, 3, 1)$ has Frobenius coordinates $(2; 3,1; 2,0)$ as seen in the Young diagram below.

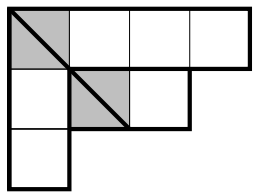
\Yboxdim{1cm}
\[\begin{tikzpicture}[scale=0.8]
\tyng(0cm,3cm,4,3,1);
\draw[fill=lightgray] (0,4) rectangle (1,3);
\draw[fill=lightgray] (1,3) rectangle (2,2);
\draw[ultra thick] (0,4)--(2,2);
\draw[ultra thick] (0,1)--(0,4)--(4,4)--(4,3)--(1,3)--(1,1)--(0,1)--(0,4);
\draw[ultra thick] (3,3)--(3,2)--(1,2);
\end{tikzpicture}\]

We then define the set
\[C_\lambda := \{-b_1 - 1/2, \dots, -b_r - 1/2, a_r + 1/2, \dots, a_1 + 1/2\}\]
and for each integer $k\geq 0,$ define $P_k(\lambda)\in \Z[\frac{1}{2}]$ by
\[P_k(\lambda) := \sum_{c \in C_\lambda} \sgn(c) c^k.\]
Finally, we define the $Q_k(\lambda)\in\Q$ by $Q_0(\lambda) := 1$ and for all $k>0,$
\[Q_k(\lambda) := \frac{P_{k - 1}(\lambda)}{(k - 1)!} + \beta_k,\]
where $\beta_k$ is defined by $\frac{z/2}{\sinh(z/2)} = \sum_{n = 0}^\infty \beta_nz^n$ (or equivalently, $\beta_k = \frac{-B_k(2^{k-1}-1)}{2^{k-1}k!}$ for all $k$).

The aforementioned theorem of Bloch and Okounkov \cite{BlochOkounkov2000} implies that the $q$-bracket (as defined in \ref{Q_Brak_Def}) of $Q_k$ is a quasimodular form for all non-negative integers $k$. We will work with 
\[\calQ_k(\tau) := 2^{k -2}(k - 1)!\left<Q_k\right>_q,\]
which is normalized so that all of the coefficients of the $q$-series besides the constant term are integral.  

\subsection{Definitions of $P_k^{(p)}(\lambda), Q_k^{(p)}(\lambda),$ and $\calQ_k^{(p)}(\tau)$} For primes $p\geq 2,$ we follow Serre and define the regularizations
\begin{align*}
P_k^{(p)}(\lambda) &:= \sum_{\begin{subarray}{c} c \in C_\lambda,\\ \gcd(2c, p) = 1\end{subarray}} \sgn(c) c^k \; \text{ for } k\geq 0,\\
Q_0^{(p)}(\lambda) &:=1-\frac{1}{p}, \text{ and}\\
Q_k^{(p)}(\lambda)  &:=  \frac{P_{k - 1}^{(p)}(\lambda)}{(k - 1)!} + \beta_k^{(p)} \; \text{ for } k>0, \\
\end{align*}
where $\beta_k^{(p)} := \beta_k(1 - p^{k - 1})$. Analogously to our normalization $\calQ_k$ of $\left<Q_k\right>_q$, we define \[\calQ_k^{(p)}(\tau) := 2^{k -2}(k - 1)!\left<Q_k^{(p)}\right>_q.\]  

\begin{rem}
By matching up conjugate partitions, we can see that $\calQ_k(\tau)$ and $\calQ_k^{(p)}(\tau)$ equal zero for odd $k$. 
\end{rem}

\section{Congruences and $p$-adic modular forms} \label{padic_section}
\subsection{Congruences} \label{cong}

Now we will show that our regularizations $\calQ_k^{(p)}(\tau)$ satisfy congruences analogous to those which are known for the Eisenstein series, proving parts (a) and (b) of Theorem~\ref{mainthm}.  We focus on weights $k$ that are not multiples of $p - 1$, as that is when the constant term of $\calQ_k^{(p)}$ is $p$-integral.  Note that this implies that $p \geq 5$.
\begin{thmpart}[part (a) of Theorem 1.1]
Let $k_1, k_2 \not \equiv 0 \pmod{p - 1}$.  Then we have that \[\calQ^{(p)}_{k_1} \equiv \calQ^{(p)}_{k_2} \pmod{p^r}\] whenever $k_1 \equiv k_2 \pmod{\phi(p^r)}$.  
\end{thmpart}

\begin{proof}
By Euler's Theorem, we get that
\[2^{k_1 - 1} P_{k_1-1}^{(p)}(\lambda) \equiv 2^{k_2 - 1}P_{k_2-1}^{(p)}(\lambda) \pmod{p^r}.\]
Since \[2^{k - 2}(k - 1)!Q_k^{(p)}(\lambda) = 2^{k - 2}P_{k - 1}^{(p)}(\lambda) + 2^{k - 2}(k - 1)! \beta_k^{(p)}\] and \[2^{k - 2}(k - 1)! \beta_k^{(p)} = -\frac{B_k^{(p)}(2^{k - 1} - 1)}{2k},\] the Kummer congruences imply that 
\[2^{k_1 - 2}(k_1 - 1)!Q_{k_1}^{(p)}(\lambda) \equiv 2^{k_2 - 2}(k_2 - 1)!Q_{k_2}^{(p)}(\lambda) \pmod{p^r}.\]  These congruences carry over to \[\calQ_k^{(p)} = \eta(\tau) \sum_{\lambda \in \calP} 2^{k - 2}(k - 1)!Q_k^{(p)}(\lambda)q^{|\lambda|-1/24}.\]
\end{proof}

\subsection{$p$-adic modular forms}
Now, we will use these congruence results to show that our regularizations $\calQ_k^{(p)}$ are $p$-adic modular forms.  As before, we focus on weights $k$ that are not multiples of $p - 1$. First, we define a $p$-adic modular form. 

\begin{df}
We say that $f = \sum a_n q^n \in \Q_p[[q]]$ is a $p$-adic modular form if there exists $f_i \in M_{k_i}$ with rational coefficients which converge uniformly to the coefficients of $f$ in $\Q_p$.  In this situation, we write $f_i \to f$.  
\end{df}

\begin{rem}
If $f_i \to f$, it can be shown that the weights $k_i$ converge in the weight space $X$.  For $p > 2$, we have that \[X = \varprojlim \Z/\phi(p^m)\Z.\] 
\end{rem}

\begin{rem}
Every level $p^n$ modular form is a level 1 $p$-adic modular form of the same weight.  This includes for instance the regularized Eisenstein series $G_k^{(p)}$ for $p \geq 5, k \not \equiv 0 \pmod{p - 1}$.  In this case, we can see that $G_{k + \phi(p^i)} \to G_k^{(p)}$. 
\end{rem}

Note that since $E_2, E_4, E_6$ are $p$-adic modular forms, all quasimodular forms are too.  Thus, since we will show that the $\calQ_k^{(p)}$'s are quasimodular in Section~\ref{jacobi}, we will have that they are $p$-adic modular forms as well.  However, we can also show this using the above congruences, which give them as $p$-adic limits of the $\calQ_k$'s.  

\begin{thmpart}[part (b) of Theorem~\ref{mainthm}]
Let $p \geq 5$ and \mbox{$k \not \equiv 0 \pmod{p - 1}$}.  Then we have that $\calQ_k^{(p)}$ is a $p$-adic modular form of weight $k$, with \[\displaystyle{\lim_{i \to \infty} \calQ_{k + \phi(p^i)} = \calQ_k^{(p)}}.\]
\end{thmpart}

\begin{proof}
For $i \geq 1$, we have that \[g_i := \calQ_{k + \phi(p^i)} \equiv \calQ_k^{(p)} \pmod{p^i},\] and so $g_i \to \calQ_k^{(p)}$ $p$-adically.  Since the $g_i$ are quasimodular, they are $p$-adic modular forms, and hence the $\calQ_k^{(p)}$ are too. 
\end{proof}

\begin{rem}
If $k \equiv 0 \pmod{p - 1}$, then $B_k^{(p)}$ is not $p$-integral and hence we do not get congruences for the constant term of $\calQ_k^{(p)}$.  However, just as with the Eisenstein series, we can renormalize so that the constant term is one. Kummer's congruences for the Bernoulli numbers imply that the resulting functions will also converge $p$-adically.  In the special case $k=0,$ the result converges $p$-adically to $1$. 
\end{rem}

\subsection{Filtration}
In addition to studying $p$-adic modular forms, we can also study modulo-$p$ modular forms.  One of the most important properties of modulo-$p$ modular forms are their filtration.  See \cite{Serre1973} for more details. 

\begin{df}
Let $p \geq 5$.  The filtration of $f \in \mathbb{F}_p[[q]]$ is denoted as $w(f)$ and is defined to be the smallest integer $k$ such that $f$ is the modulo $p$ reduction of a modular form of weight $k$ and level 1 with coefficients in $\Q \cap \Z_p$.  
\end{df}

\begin{thmpart}[part (c) of Theorem~\ref{mainthm}]
If $p > k$ then the modulo $p$ filtration of $\calQ_k^{(p)}$ (and $\calQ_k$) is $k(p + 1)/2$.  
\end{thmpart}
\begin{proof}
First, note that $\calQ_k \equiv \calQ_k^{(p)} \pmod{p}$ since they only differ modulo higher powers of $p$.  Thus they must have the same filtration.

We may write $\calQ_k$ as a polynomial of degree $k/2$ in $G_2,$ and Theorem 2 of \cite{Zagier2015} gives us the leading coefficient: \[\calQ_{k}=\frac{(k - 1)!!\; 8^{k/2 - 1}}{k/2} G_2^{k/2} + \text{lower degree terms in $G_2$},\] where $(k-1)!! = 1\times 3 \times \cdots \times (k-1)$. Note that since $p>k,$ the leading coefficient $\frac{(k - 1)!!\; 8^{k/2 - 1}}{k/2}$ must be nonzero modulo $p$.  Also, it is well-known that $w(G_2^i) = iw(G_2) = i(p + 1)$ for all $i \geq 1$.  Thus the filtration of $\frac{(k - 1)!!\; 8^{k/2 - 1}}{k/2} G_2^{k/2}$ is $k(p+1)/2,$ and the filtrations of the ``lower degree terms'' are strictly smaller. It follows that the filtration of $\calQ_k$ is $k(p+1)/2$, as desired.
\end{proof}

\section{Jacobi forms and the quasimodularity of $\calQ_k^{(p)}$} \label{jacobi}
\newcommand{\SL}{\mathrm{SL}}
\newcommand{\im}{\mathrm{im}}

In this section we will show that $\calQ_k^{(p)}(\tau)$ is quasimodular for every prime $p$ and non-negative integer $k$. We define certain auxiliary functions $F(z,\tau)$ and $F^{(p)}(z,\tau)$ which are generating functions for the $q$-brackets $\calQ_k(\tau)$ and $\calQ_k^{(p)}(\tau)$, and we show these functions are Jacobi forms. We then make use of the theory of Jacobi forms to prove our quasimodularity result.

\subsection{Definitions of $F(z,\tau)$ and $F^{(p)}(z,\tau)$}
In order to better understand $\calQ_k^{(p)}(\tau)$, we first define the function $F:\C\times \H \rightarrow \C$ and its regularization $F^{(p)}$ by
\begin{eqnarray}
F(z,\tau) &:=& \frac{1}{2}\sum_{k\geq 0}\left<Q_k\right>_q(4\pi i z)^{k-1} \; \text{ and}\label{deff}\\
F^{(p)}(z,\tau) &:=& \frac{1}{2}\sum_{k\geq 0}\left<Q_k^{(p)}\right>_q(4\pi i z)^{k-1}, \label{deffp}
\end{eqnarray}
where we set $q=e^{2\pi i\tau}$ and $\zeta=e^{2\pi iz}.$

We can describe the function $F^{(p)}(z,\tau)$ in terms of $F(z,\tau)$ as follows. 

\begin{pr}\label{pr2.1}
Using the notation given above, we have that \[F^{(p)}(z,\tau) = F(z,\tau) - \frac{1}{p}\sum_{j=0}^{p-1}F(z+j/p,\tau).\]
\end{pr} 
\begin{proof}
First we claim that
\begin{eqnarray}
F(z; \tau) &=&\frac{1}{2(\zeta-\zeta^{-1})}+\frac{\eta(\tau)}{2}\sum_{\lambda \in \calP}\sum_{c\in C_\lambda}\sgn(c)\zeta^{2c}q^{|\lambda|-1/24} \; \text{ and} \label{pr2.1claim1}\\
F^{(p)}(z; \tau) &=& \frac{1}{2}\left(\frac{1}{\zeta-\zeta^{-1}}-\frac{1}{\zeta^p-\zeta^{-p}}\right) \notag \\ & &+\frac{\eta(\tau)}{2}\sum_{\lambda \in \calP}\sum_{\begin{subarray}{c} c \in C_\lambda,\\ \gcd(2c, p) = 1  \end{subarray}}\sgn(c)\zeta^{2c}q^{|\lambda|-1/24}. \label{pr2.1claim2}
\end{eqnarray}

\noindent Equation \eqref{pr2.1claim1} follows by differentiating and evaluating at $z=0;$ for all integers $k\geq 1$ we have 
\begin{align*}
\left(\frac{d}{dz}\right)^{k-1}\left(\frac{\eta(\tau)}{2}\sum_{\lambda \in \calP}\right. & \left.\left.\sum_{c\in C_\lambda}\sgn(c)\zeta^{2c}q^{|\lambda|-1/24}\right)\right|_{z=0}\\
&= \frac{\eta(\tau)}{2}\sum_{\lambda \in \calP}\sum_{c\in C_\lambda}\sgn(c)(4\pi ic)^{k-1}q^{|\lambda|-1/24}\\
&= \frac{\eta(\tau)}{2}\sum_{\lambda \in \calP}(4\pi i)^{k-1}P_{k-1}(\lambda)q^{|\lambda|-1/24}.
\end{align*}
Using the equation $\frac{4\pi iz}{\zeta-\zeta^{-1}} = \sum_{n=0}^\infty\beta_n(4\pi iz)^n,$ we see that
\begin{align*}
\left.\left(\frac{d}{dz}\right)^{k-1} \left(\frac{1}{2(\zeta-\zeta^{-1})} -\frac{1}{8\pi i z}\right)\right|_{z=0}
&= \left.\left(\frac{d}{dz}\right)^{k-1} \left(\frac{1}{8\pi i z}\sum_{n=0}^\infty\beta_n(4\pi iz)^n -\frac{1}{8\pi i z}\right)\right|_{z=0}\\
&= \frac{1}{2}(k-1)!(4\pi i)^{k-1}\beta_{k}.
\end{align*}

\noindent Together these imply that 
\begin{align*}
\frac{1}{2(\zeta-\zeta^{-1})}+\frac{\eta(\tau)}{2}&\sum_{\lambda \in \calP}\sum_{c\in C_\lambda}\sgn(c)\zeta^{2c}q^{|\lambda|-1/24}\\
&=  \frac{1}{2}\sum_{k\geq 0}\left<Q_k\right>_q(4\pi i z)^{k-1}\\
&=F(z,\tau).
\end{align*}

\noindent Equation \eqref{pr2.1claim2} holds by a similar argument.

A direct calculation using \eqref{pr2.1claim1} shows that
\begin{align*}
F(z,\tau) - \frac{1}{p}\sum_{j=0}^{p-1}F(z+j/p,\tau) &= \frac{1}{2}\left(\frac{1}{\zeta-\zeta^{-1}}-\frac{1}{\zeta^p-\zeta^{-p}}\right)\\ &+ \frac{\eta(\tau)}{2}\sum_{\lambda \in \calP}\sum_{\begin{subarray}{c} c \in C_\lambda,\\ \gcd(2c, p) = 1  \end{subarray}}\sgn(c)\zeta^{2c}q^{|\lambda|-1/24}\\
&=F^{(p)}(z; \tau)
\end{align*}
\end{proof}

\subsection{Jacobi forms}
It turns out that $F(z;\tau)$ and $F^{(p)}(z;\tau)$ are meromorphic Jacobi forms. In order to define Jacobi forms, we begin by defining the Petersson slash operator for Jacobi forms. Given a matrix $\gamma=\begin{pmatrix}a&b\\c&d\end{pmatrix}\in \SL_2(\Z)$ and a function $\phi(z;\tau):\C\times \H\to \C,$ define the weight $k$ index $m$ slash operator by

\[\left(\phi|_{k,m}\gamma \right) (z;\tau):=(c\tau+d)^{-k} e\left(\frac{-cmz^2}{c\tau+d}\right)\phi\left(\frac{z}{c\tau+d};\frac{a\tau+b}{c\tau+d}\right),\]
where $e(\alpha) := e^{2\pi i\alpha}$.

Jacobi forms are invariant under the action of matrices with respect to the slash operator and elliptic transformations.
\begin{df}\label{JacobiDef}
If $k$ and $m$ are integers, a holomorphic Jacobi form of weight $k$ and index $m$ for some subgroup $\Gamma$ of $\SL_2(\Z)$ is a holomorphic function 
\[\phi(z;\tau):\C\times\H\to \C\]
which satisfies the following properties:
\begin{enumerate}
\item for every $\gamma \in \Gamma,$
\[\left(\phi|_{k,m}\gamma \right)(z; t)=\phi(z;\tau),\]
\item for every pair of integers $a$ and $b$, 
\[\phi\left(z+a\tau+b;\tau\right)=e(-m(a^2\tau+2az))\phi(z;\tau),\]
\item and the function $\phi(z;\tau)$ has a Fourier expansion of the form 
\[\phi(z;\tau)=\sum_{\substack{n,r\in \Z\\n\geq \frac{r^2}{4m}}}c(n,r)q^n\zeta^r.\]
\end{enumerate}
\end{df}
We refer to the variable $z$ in the definition above as the {elliptic} variable and to $\tau$ as the {modular} variable.

A meromorphic Jacobi form is a function which satisfies properties (1), (2), and (3) in the definition above, but is required only to be meromorphic in the elliptic variable and weakly holomorphic in the modular variable--that is for fixed $z$, the function $\tau\mapsto \phi(z;\tau)$ is holomorphic on $\H$ and meromorphic at the cusps of $\H/\Gamma$.
 For more details on Jacobi forms, see \cite{EichlerZagier1985}.
\begin{pr}
The functions $F(z;\tau)$ and $F^{(p)}(z;\tau)$ are meromorphic Jacobi forms of weight $1$ and index $-2$ for $\SL_2(\Z)$ and $\Gamma_0(p^2)$ respectively, with simple poles at the points $z\in \frac{1}{2}\Z \oplus\frac{\tau}{2}\Z$, and $z\in \frac{1}{2p}\Z \oplus\frac{\tau}{2}\Z$ respectively.
\end{pr}

\begin{proof}
In Section 6 of \cite{BlochOkounkov2000}, Bloch and Okounkov prove that \begin{equation} \label{BOThm6.5}
F(z, \tau) = \frac{(q)_\infty^{2}}{(\zeta-\zeta^{-1})(q\zeta^2)_\infty(q\zeta^{-2})_\infty}=\frac{\eta(q)^3}{2\theta_1(2z;q)},
\end{equation}
where $\theta_1(z;\tau)$ is the standard theta function 
\[\theta_1(z;\tau) := \sum_{n \in \Z}(-1)^nq^{(n+1/2)^2/2}\zeta^{n+1/2}.\]

Although the definition given above only allows for Jacobi forms of integer weight and index, it can be modified to allow Jacobi forms of half integer weight and index, although care must be taken with the square root.
The function $\theta_1$ defined above is a holomorphic function with simple zeros at the points $z\in \Z \oplus\tau \Z$, and transforms like a Jacobi form of weight $1/2$ and index $1/2$ for $\SL_2(\Z)$ with a multiplier. In particular we have that 

\begin{align*}
\theta_1(z/\tau;-1/\tau)=-i(-i\tau)^{1/2}e\left(\frac{z^2}{2\tau}\right) \theta_1(z;\tau).
\end{align*}

Taking into account the change of variable, these facts together imply the proposition for $F(z;\tau)$.

To prove the proposition for $F^{(p)}(z;\tau),$ we use Proposition \ref{pr2.1} and consider the individual components $F(z+j/p;\tau)$. Using the Jacobi form transformation laws we find that if $\begin{pmatrix}a&b\\c&d\end{pmatrix}\in \Gamma_0(p^2),$ then 

\begin{align*}
  F&\left(\frac{z}{c\tau+d}+j/p;\frac{a\tau+b}{c\tau+d}\right)\\
& \ \ =(c\tau+d) \ e\left(-2c\frac{\left(z+j(c\tau+d)/p\right)^2}{c\tau+d}\right) \ F\left(z+\frac{jc}{p}\tau+\frac{jd}{p};\tau\right)\\
& \ \ =(c\tau+d) \ e\left(-2c\frac{z^2}{c\tau+d}-2\frac{j^2cd}{p^2}\right) \ F\left(z+\frac{jd}{p};\tau\right).
\end{align*}
Since $p^2$ divides $c$, the action of the matrix simply permutes these terms, leaving only the expected automorphy factor, $(c\tau+d) \ e\left(-2c\frac{z^2}{c\tau+d}\right).$

\end{proof}

\subsection{Showing quasimodularity}
In this section, we will prove the following. 
\begin{thmpart}[part (d) of Theorem~\ref{mainthm}]
Let $p \geq 5$ be prime.  Then we have that $\calQ_k^{(p)} \in \widetilde{M}_k(p^2)$. 
\end{thmpart}
We have defined the functions $F$ and $F^{(p)}$ so that the $q$-brackets and regularized $q$-brackets arise from derivatives of these functions. Although derivatives of Jacobi forms are not generally Jacobi forms themselves, there is a certain differential operator $Y_{m}$ which preserves the modularity properties, but sacrifices holomorphicity. Let $Y_{m}$ be the Jacobi raising operator defined as follows (see \cite[pg. 43]{BerndtSchmidt}, or \cite[Def. 2.5]{CR}):
\[Y_m(\phi) := \frac{1}{2\pi i}\frac{\partial \phi}{\partial z}+2m\frac{\text{Im } z}{\text{Im }\tau}\phi\] for all meromorphic $\phi(z;\tau):\C\times \H\to \C.$ 
\begin{pr}[Berndt, Schmidt]
The operator $Y_{m}$ commutes with the action of the slash operator but increases the weight by $1$, so that if \newline$\gamma=\begin{pmatrix}a&b\\c&d\end{pmatrix}\in \SL_2(\Z),$ and $\phi(z;\tau)$ is any real-analytic function $\phi(z;\tau):\C\times \H\to \C,$ then 
\[Y_{m} \left(\phi|_{k,m}\gamma \right)  = \left(Y_{m} \phi\right)|_{k+1,m}\gamma.\]
Therefore, if $\phi$ transforms like a Jacobi form of weight $k$ and index $m$, then $Y_{m} (\phi)$ transforms like a Jacobi form of weight $k+1$ and index $m$.
\end{pr}
\begin{proof}
Let $\phi'(z;\tau):=\frac{1}{2\pi i}\frac{\partial}{\partial z}\phi(z;\tau).$ Using the definitions above we have 

\begin{multline*}
Y_{m} \left(\phi|_{k,m}\gamma \right) = (c\tau+d)^{-k} e\left(\frac{-cmz^2}{c\tau+d}\right) \left[ \phi'\left(\frac{z}{c\tau+d};\frac{a\tau+b}{c\tau+d}\right)(c\tau+d)^{-1} \right.\\
\left.+\left(-\frac{2cmz}{c\tau+d} +2m\frac{\text{Im } z}{\text{Im }\tau}\right)\phi\left(\frac{z}{c\tau+d};\frac{a\tau+b}{c\tau+d}\right)\right].
\end{multline*}

\noindent A short (but slightly messy) calculation shows that 
\begin{align*}
\frac{\text{Im } \frac{z}{c\tau+d}}{\text{Im }\frac{a\tau+b}{c\tau+d} }= -cz+(c\tau+d)\frac{\text{Im } z}{\text{Im }\tau},
\end{align*}
which allows us to turn the above equation into 

\begin{equation*}\begin{split}
Y_{m} \left(\phi|_{k,m}\gamma \right) &= (c\tau+d)^{-k-1} e\left(\frac{-cmz^2}{c\tau+d}\right) \left[ \phi'\left(\frac{z}{c\tau+d};\frac{a\tau+b}{c\tau+d}\right) \right.\\
&\qquad \left.+2m\frac{\text{Im } \frac{z}{c\tau+d}}{\text{Im }\frac{a\tau+b}{c\tau+d} }\phi\left(\frac{z}{c\tau+d};\frac{a\tau+b}{c\tau+d}\right)\right]\\
&= \left(Y_{m} \phi\right) |_{k+1,m}\gamma
\end{split}\end{equation*}

\noindent as desired.
\end{proof}

\newcommand{\Res}{\operatornamewithlimits{Res\ }}
Eichler and Zagier show that the Taylor coefficients with respect to the elliptic variable of a holomorphic Jacobi forms are quasimodular forms \cite{EichlerZagier1985}. The idea is as follows: suppose $\phi(z;\tau)$ is any function which is invariant under the slash operator $|_{k,m}.$ If $\phi(0,\tau)$ is defined, than the transformation laws imply this function in $\tau$ transforms like a weight $k$ modular form. Thus, if $\Phi(z;\tau)$ is a holomorphic Jacobi form of index $m$, then $Y^n_m(\Phi)(0;\tau)$ transforms like a modular form. It is not difficult then to see that the holomorphic component $\left(\frac{1}{2\pi i}\frac{\partial}{\partial z}\right)^n(\Phi)(0;\tau)$ must be quasimodular.

Unfortunately, $F$ and $F^{(p)}$ both have poles at $z=0.$ We can work around this problem by taking a residue at $z=0$, however we must be careful how we do this. The following procedure is similar to that followed by other authors, including Bringmann and Folsom \cite[Section 3]{BringmannFolsom} and Olivetto \cite{Olivetto}. 
Suppose $G(z)$ is any function which is real-analytic near $z=0$ with a singularity of at most finite order at $0$ (i.e there is some positive integer $j$ such that $|z|^{2j} G(z)$ is real-analytic in a neighborhood of $0$). Then $G(z)$ can be written as a Laurant series in $z$ and $\overline z,$ or in polar coordinates by setting $z=r~e^{2 \pi i \theta}$ and $\overline z=r~e^{-2 \pi i \theta}.$ Then we define $\displaystyle \Res_{z=0}\frac{1}{z}G(z)$ to be the limit of the integral

\[
\Res_{z=0}\frac{1}{z}G(z):=\lim_{r\to 0}\frac{1}{2\pi i}\int_{0}^1 \frac{G(re^{2\pi i \theta})}{r~e^{2 \pi i \theta}}r ~d e^{2 \pi i \theta} =\lim_{r\to 0}\int_{0}^1 G(re^{2\pi i \theta})d \theta
\]
whenever this limit exists. This limit will converge so long as no term in the Laurant series expansion for $G(z)$ at $0$ is a negative power of $z\overline z.$ In our case the denominators will only have powers of the holomorphic variable. 

\begin{pr}\label{residue}
Suppose $\phi(z;\tau):\C\times \H\to \C$ is a function which is real-analytic in $z$ near $z=0$ with at most a singularity of finite order at $z=0$ and which transforms like a weight $k$ index $m$ Jacobi form on $\Gamma$. If the residue 
\[\Res_{z=0}\frac{1}{z}e\left(mz\frac{\im z}{\im \tau}\right)\phi(z;\tau)\]

\noindent exists, then it transforms in $\tau$ like a modular form of weight $k$ for $\Gamma$.
\end{pr}
\begin{proof}
We begin by noting that the function $\Phi(z;\tau):= e\left(mz\frac{\im z}{\im \tau}\right)\phi(z;\tau)$ has a particularly clean transformation law: 
\[\Phi\left(\frac{z}{c\tau+d};\frac{a\tau+b}{c\tau+d}\right)=(c\tau+d)^{k}\Phi(z;\tau).\]

\noindent This follows from the identity 
\[\frac{cz^2}{c\tau+d}+\frac{z}{c\tau+d}\frac{\im \frac{z}{c\tau+d}}{\im \frac{a\tau+b}{c\tau+d}}=z\frac{\im z}{\im \tau}.\]

\noindent As above, set $z=re^{2\pi i \theta},$ and let 
\[R(\tau):=\text{Res}_{z=0}\frac{1}{z}\Phi(z;\tau)=\lim_{r\to 0} \int_{0}^1\Phi(z;\tau)d \theta.\]

\noindent Then if we let $z'=z(c\tau+d)=r'e^{2\pi i \theta'}$, we have that  
\begin{align*}
R\left(\frac{a\tau+b}{c\tau+d}\right)&=\lim_{r\to 0} \int_{0}^1\Phi\left(z;\frac{a\tau+b}{c\tau+d}\right)d \theta\\
&=(c\tau+d)^{k}\lim_{r'\to 0} \int_{0}^1\Phi\left(z';\tau\right)d \theta'\\
&=(c\tau+d)^{k}R(\tau).
\end{align*}
\end{proof}

We are now ready to prove part (d) of Theorem~\ref{mainthm}.

\begin{proof}[Proof of Theorem~\ref{mainthm}(d)]Let $\widetilde \calQ_k(\tau)$ and $\widetilde \calQ^{(p)}_k(\tau)$ be the results of applying Proposition \ref{residue} to the functions $Y_{-2}^{k-1}F(z;\tau)$ and $Y_{-2}^{k-1}F^{(p)}(z;\tau)$ respectively. The definition of $Y_{-2}$ implies that $\widetilde \calQ_k(\tau)$ is a function of $y:=\im \tau$ and $\tau$, and can be written in the form 
\[\widetilde \calQ_k(\tau)=\sum_{j=0}^{k-1}\left(\frac{1}{y}\right)^jf_j(\tau),\]
as a polynomial in $1/y$ whose coefficients, $f_m,$ are holomorphic $q$-series. In fact we see that $f_0=2^{k-1}\calQ_k(\tau),$ and each of the remaining $f_m$ can be written in terms derivatives of brackets $\calQ_\ell(\tau)$ with $\ell<k.$ Since $\widetilde \calQ_k(\tau)$ transforms like a modular form, it follows that $\calQ_k(\tau) \in \widetilde M_{k}$. A nearly identical argument holds for $\widetilde \calQ_k^{(p)}(\tau);$ since $\widetilde \calQ_k^{(p)}(\tau)$ transforms like a modular form for $\Gamma_0(p^2),$ we must have that $\calQ_k^{(p)}(\tau)\in \widetilde M_{k}(p^2)$. This completes the proof of Theorem $\ref{mainthm}$ (d).
\end{proof}

\subsection{Finding an explicit expression}
Just as we can write $G_k^{(p)}$ in terms of $G_k$, we can write $\calQ_k^{(p)}$ in terms of $\calQ_k$.  However, there is an extra correction term
\begin{thmpart}[part (e) of Theorem~\ref{mainthm}]
Let $p \geq 5$ be prime.  We have that
\[\calQ_k^{(p)}(\tau) = \calQ_k(\tau) - p^{k - 1}\calQ_k(p^2 \tau) - p^{k - 1}f_k^{(p)},\] where \[f_k^{(p)}:=-\sum_{\substack{n\geq1\\(n,p)=1}}\sum_{M\geq 0}(-1)^n(2M+1)^{k-1}q^{\frac{n^2+np(2M+1)}{2}}.\]
\end{thmpart}
\begin{proof}
The function $F(z;\tau)$ is closely related to the partition crank generating function, which is given \cite{AndrewsGarvan1988, Zwegers2010} by

\begin{equation*}
C(z;\tau):=\prod_{n\geq1}\frac{1-q^n}{(1-\zeta q^n)(1-\zeta^{-1} q^n)}=\frac{1-\zeta}{(q)_\infty}\sum_{n\in \Z}\frac{(-1)^nq^{\frac{n(n+1)}{2}}}{1-\zeta q^n}.
\end{equation*}
\noindent Using this identity, we obtain an alternative expression for $F,$
\begin{align*}
F(z;\tau)&=-\frac{1}{2} \zeta\sum_{n\in \Z}\frac{(-1)^nq^{\frac{n(n+1)}{2}}}{1-\zeta^2 q^n}\\
&=\frac{1}{2} ~ \frac{1}{\zeta-\zeta^{-1}}-\frac{1}{2}\sum_{n\geq1}\sum_{m\geq 0}(-1)^n(\zeta^{2m+1}-\zeta^{-2m-1})q^{\frac{n(n+1)}{2}+m \, n}.
\end{align*}
Using Proposition \ref{pr2.1}, we have that the difference $F(z;\tau)-F^{(p)}(z;\tau)$ simply isolates the coefficients of powers of $\zeta$ divisible by $p$. When $p$ is an odd prime, we have that $F(z;\tau)-F^{(p)}(z;\tau)$ is given by
\[
\frac{1}{2} ~ \frac{1}{\zeta^p-\zeta^{-p}} -\frac12\sum_{n\geq1}\sum_{M\geq 0}(-1)^n(\zeta^{p(2M+1)}-\zeta^{-p(2M+1)})q^{\frac{n^2+np(2M+1)}{2}}.
\]
Here we have made the substitution $2m+1=p(2M+1)$ when $p$ divides $2m+1.$ If we separate the terms where $p$ divides $n$, $F(z;\tau)-F^{(p)}(z;\tau)$ becomes

\begin{equation}\label{diffexp}
F(pz;p^2\tau)-\frac12\sum_{\substack{n\geq1\\(n,p)=1}}\sum_{M\geq 0}(-1)^n(\zeta^{p(2M+1)}-\zeta^{-p(2M+1)})q^{\frac{n^2+np(2M+1)}{2}}.
\end{equation}

Let $\tilde F^{(p)}(z;\tau):=F(z;\tau)-F^{(p)}(z;\tau)-\frac{1}{8p\pi i z},$ so that we have removed the pole at $z=0.$ Then if $k$ is even we use equation (\ref{diffexp}) above to find have that
\begin{align*}
\left.\left(\frac{1}{2\pi i}\frac{d}{dz}\right)^{k-1}\tilde F(z;\tau)\right|_{z=0}&=\calQ_k(\tau)-\calQ^{(p)}(\tau) &\\and
&=p^{k-1}\calQ_k(p^2\tau)+p^{k-1}f_k^{(p)},
\end{align*}

\noindent where 
\[f_k^{(p)}:=-\sum_{\substack{n\geq1\\(n,p)=1}}\sum_{M\geq 0}(-1)^n(2M+1)^{k-1}q^{\frac{n^2+np(2M+1)}{2}}.\]
Note that since $(n,p)=1,$ then for the exponent $N=\frac{n^2+np(2M+1)}{2}$ we have that $\left(\frac{N}{p}\right)=\left(\frac{2}{p}\right).$
\end{proof}


\begin{backmatter}



\section*{Acknowledgements}
The authors began jointly discussing this work at the Spring School on Characters of Representations and Modular Forms held at the Max Planck Institute in Bonn, Germany in March 2015 and are grateful for the good hospitality and excellent conference.  The authors would also like to thank Don Zagier for his inspiring work, and Ken Ono and the referee for their helpful comments.  The second and third authors thank the National Science Foundation for its support. 

\bibliographystyle{bmc-mathphys} 
\bibliography{refs}      

\end{backmatter}
\end{document}